\documentclass[10pt]{article}

\usepackage{geometry, enumerate, amsmath, amssymb, amsthm, amscd, mathrsfs, hyperref, graphicx, color, enumerate, float, verbatim, colonequals, url}

\newtheorem{thm}{Theorem}[section]
\newtheorem{cor}[thm]{Corollary}

\newtheorem{lem}[thm]{Lemma}

\newtheorem{prop}[thm]{Proposition}
\newtheorem{ques}[thm]{Question}

\theoremstyle{definition}
\theoremstyle{definition}\newtheorem{Def}[thm]{Definition}
\theoremstyle{definition}
\theoremstyle{definition}\newtheorem{Not}[thm]{Notation}

\newcommand{\N}{\mathbb{N}}

\newcommand{\F}{\mathbb{F}}

\renewcommand{\le}{\leqslant}
\renewcommand{\ge}{\geqslant}

\newcommand{\mcJ}{\mathcal{J}}

\newcommand{\msJ}{\mathscr{J}}

\newcommand{\msS}{\mathscr{S}}
\newcommand{\msT}{\mathscr{T}}
\newcommand{\mcC}{\mathcal{C}}

\newcommand{\olR}{\overline{R}}

\newcommand*\xbar[1]{%
   \hbox{%
     \vbox{%
       \hrule height 0.5pt 
       \kern0.5ex
       \hbox{%
         \kern-0.1em
         \ensuremath{#1}%
         \kern-0.1em
       }%
     }%
   }%
}

\title{Covering Numbers of Commutative Rings}

\date{\today}

\author{Eric Swartz\footnote{Department of Mathematics, William \& Mary, Williamsburg, VA 23187, USA, easwartz@wm.edu}
\and 
Nicholas J. Werner\footnote{Department of Mathematics, Computer and Information Science, State University of New York College at Old Westbury, Old Westbury, NY 11560, USA, wernern@oldwestbury.edu}
}

\begin{document}




\maketitle 

\thispagestyle{empty}

\begin{abstract}
A cover of a unital, associative (not necessarily commutative) ring $R$ is a collection of proper subrings of $R$ whose set-theoretic union equals $R$. If such a cover exists, then the covering number $\sigma(R)$ of $R$ is the cardinality of a minimal cover, and a ring $R$ is called $\sigma$-elementary if $\sigma(R) < \sigma(R/I)$ for every nonzero two-sided ideal $I$ of $R$. In this paper, we show that if $R$ has a finite covering number, then the calculation of $\sigma(R)$ can be reduced to the case where $R$ is a finite ring of characteristic $p$ and the Jacobson radical $J$ of $R$ has nilpotency 2.  Our main result is that if $R$ has a finite covering number and $R/J$ is commutative (even if $R$ itself is not), then either $\sigma(R)=\sigma(R/J)$, or $\sigma(R)=p^d+1$ for some $d \ge 1$.  As a byproduct, we classify all commutative $\sigma$-elementary rings with a finite covering number and characterize the integers that occur as the covering number of a commutative ring.
Keywords:
\end{abstract}

\section{Introduction}\label{Intro section}

The goal of this paper is to study how a ring can be expressed as a union of proper subrings. The motivation for this comes from the corresponding problem for groups. Given a group $G$, we say that $G$ is \textit{coverable} if there exists a collection $\mathcal{H}$ of proper subgroups of $G$ such that $G = \bigcup_{H \in \mathcal{H}} H$. If such a set $\mathcal{H}$ exists, then it is called a \textit{cover} of $G$, and the \textit{covering number} of $G$, denoted by $\sigma(G)$, is the cardinality of a minimal cover. 

It is easy to see that a group if coverable if and only if it is not cyclic. Furthermore, it is well known that no group is the union of two proper subgroups, so $\sigma(G) \ge 3$ for any coverable group $G$, and in 1926 Scorza \cite{Scorza} characterized groups with covering number equal to $3$.  Cohn \cite{Cohn} proved that every integer of the form $p^d + 1$, where $p$ is a prime and $d$ is a positive integer, occurs as a covering number, and Tomkinson \cite{Tom} subsequently proved that the covering number of a solvable group is always of the form $p^d + 1$.  On the other hand, covering numbers of nonsolvable groups are far less predictable, making it difficult to determine precisely which integers are covering numbers of groups.  Recently, the integers less than or equal to $129$ that are covering numbers of groups were determined in \cite{GaronziKappeSwartz}. We refer the reader to the survey \cite{Kappe} and the introduction of \cite{GaronziKappeSwartz} for summaries of the extensive literature on covering numbers of groups and related problems.

We will investigate similar covering problems for rings. In this paper, all rings are associative and unital. For any ring $R$, $\msJ(R)$ denotes the Jacobson radical of $R$, and often we let $J = \msJ(R)$. For a prime power $q$, $\F_q$ denotes the finite field with $q$ elements. While a ring must contain unity, subrings need not contain a multiplicative identity. That is, given a ring $R$ and a subset $S \subseteq R$, $S$ is a subring of $R$ if it is a group under addition and is closed under multiplication. A \textit{cover} of $R$ is a collection $\mathcal{C}$ of proper subrings of $R$ such that $R = \bigcup_{S \in \mathcal{C}} S$. If a cover exists, then $R$ is \textit{coverable}, and we define $\sigma(R)$ to be the cardinality of a minimal cover. The reason that we permit a subring to lack a unit element is because it allows for a larger class of subrings to be used in ring covers; in particular, ideals of $R$ can be included in covers when necessary. In many cases, however, we will construct covers that consist of maximal subrings of $R$, and such subrings often contain $1_R$.

\begin{lem}\label{1 in max lem}\cite[Lem.\ 2.1]{PeruginelliWerner}
Let $R$ be a ring with unity. Let $M$ be a maximal subring of $R$. Then, either $1_R \in M$; or, $M$ is a maximal two-sided ideal of $R$ such that $R/M \cong \F_p$ for some prime $p$.
\end{lem}

The literature on covering numbers of rings is not as capacious as that for groups. Nevertheless, in recent years a number of papers have appeared on the topic, and covering numbers have been computed for several classes of rings. All rings (with or without unity) having covering number 3 were classified in \cite{LucMar}; a recent preprint \cite{Cohen} considers the same question for covering number 4. Formulas for $\sigma(R)$ are known when $R$ is a matrix ring over a finite field \cite{LucMarP, Crestani}, a direct product of finite fields \cite{Werner}, a finite semisimple ring \cite{PeruginelliWerner}, or a $2 \times 2$ upper triangular matrix ring over a finite field \cite{CaiWerner}.

It is known that no group has covering number 7 \cite{Tom} or 11 \cite{DetLuc}. An obvious question for rings is whether there exists $n \ge 3$ such that no ring has covering number $n$. There do exist rings with covering number 7 (the matrix ring $M_2(\F_3)$) and 11 (the matrix ring $M_2(\F_4)$); in fact, examples of rings $R$ with $\sigma(R)=n$ are known for all $3 \le n \le 12$. Thus, if such an $n$ exists, then it must be at least 13. 

\begin{ques}
Does there exist a (unital, associative) ring $R$ with $\sigma(R)=13$?
\end{ques}

While we are not able to fully resolve this question, we will answer it in the negative for commutative rings, and more generally for any ring $R$ such that $R/\msJ(R)$ is commutative, where $\msJ(R)$ is the Jacobson radical of $R$. To accomplish this, we examine how a cover of $R$ relates to covers of subrings and residue rings of $R$. If $I$ is a two-sided ideal of $R$, then one may lift a cover of $R/I$ to $R$. Consequently, it is always true that $\sigma(R) \le \sigma(R/I)$. The rings for which this inequality is always strict deserve special attention.

\begin{Def}
A ring $R$ will be called \textit{$\sigma$-elementary} if $R$ is coverable and $\sigma(R) < \sigma(R/I)$ for every nonzero two-sided ideal $I$ of $R$.
\end{Def}

It is clear from this definition that if $R$ is coverable, then there is an ideal of $R$ such that $R/I$ is $\sigma$-elementary and $\sigma(R) = \sigma(R/I)$. Also, if an integer $n$ occurs as the covering number of a ring, then there must be a $\sigma$-elementary ring $R$ such that $\sigma(R)=n$. Hence, $\sigma$-elementary rings are quite useful for answering questions about covering numbers. In Theorem \ref{thm:comm_sigma}, we give a full classification of commutative $\sigma$-elementary rings with a finite covering number. This is used later to prove our main result.

\begin{thm}\label{thm:main}
Let $R$ be a finite $\sigma$-elementary ring with unity, let $J$ be the Jacobson radical of $R$, and assume that $R/J$ is commutative. If $J \ne \{0\}$, then $\sigma(R) = p^d + 1$, where $p$ is the characteristic of $R$ and $d$ is a positive integer.
\end{thm}

Theorem \ref{thm:main} can be viewed as an analog for rings of Tomkinson's result about solvable groups. If $J = \{0\}$ and $R=R/J$ is commutative, then $R$ is a direct product of finite fields, in which case the covering number can be found by the formulas in \cite{Werner} (summarized in Section \ref{Commutative section} below). In either case, the covering number of the ring cannot be 13.

\begin{cor}\label{cor:13}
Let $R$ be a (unital, associative) ring such that $R/\msJ(R)$ is commutative. Then, $\sigma(R) \ne 13$. 
\end{cor}

The paper is structured as follows. In Section \ref{Basic section}, we collect the basic properties of covers, covering numbers, and $\sigma$-elementary rings that will be used throughout the article. In Section \ref{Reduction section}, we prove (Theorem \ref{full reduction thm}) that if a ring $R$ has a finite covering number, then to compute $\sigma(R)$ we may assume that $R$ is a finite ring of characteristic $p$ with Jacobson radical of nilpotency 2. Section \ref{Commutative section} is devoted to the promised classification of commutative $\sigma$-elementary rings (Theorem \ref{thm:comm_sigma}), which allows us to fully classify the possible covering number of a commutative ring (Corollary \ref{Full commutative case}). In Section 5, we generalize the results of Section 4 to the case where $R/\msJ(R)$---but not necessarily $R$ itself---is commutative. The paper closes with the proofs of Theorem \ref{thm:main} and Corollary \ref{cor:13}.

\section{Basic Properties}\label{Basic section}

Recall from the introduction that a ring $R$ is coverable if there exists a collection $\mcC$ of proper subrings of $R$ such that $R = \bigcup_{S \in \mcC} S$, and $\sigma(R)$ is the size of a minimal cover (if one exists). We list below several basic, but useful, observations about covers and covering numbers. The proofs of these are straightforward, and we will use these properties without comment for the remainder of this work.

\begin{lem}\label{lem:basics}
Let $R$ be a ring with unity.
\begin{enumerate}[(1)]
\item $R$ is coverable if and only if $R$ cannot be generated (as a ring) by a single element.
\item If $R$ is noncommutative, then $R$ is coverable.
\item For any two-sided ideal $I$ of $R$, a cover of $R/I$ can be lifted to a cover of $R$. Hence, $\sigma(R) \le \sigma(R/I)$.
\item If each proper subring of $R$ is contained in a maximal subring, then we may assume that any minimal cover of $R$ consists of maximal subrings.
\end{enumerate}
\end{lem}

Next, we say that $R$ is a $\sigma$-elementary ring if $\sigma(R) < \sigma(R/I)$ for all nonzero two-sided ideals of $R$. Evidently, if $R$ is coverable but is not $\sigma$-elementary, then there exists a residue ring $\olR$ of $R$ such that $\sigma(R) = \sigma(\olR)$ and $\olR$ is $\sigma$-elementary. Hence, we can determine $\sigma(R)$ if we know the covering number for each $\sigma$-elementary residue ring of $R$. Also, knowing that $\sigma(R) < \sigma(R/I)$ for some ideal $I \subseteq R$ may provide information on the subrings needed to cover $R$.

\begin{lem}\label{Sigma elementary lemma}
Let $R$ be a coverable ring, let $\mcC$ be a minimal cover of $R$, and let $S$ be a subring of $R$ such that $R=S \oplus I$ for some two-sided ideal $I$ of $R$. If $\sigma(R) < \sigma(R/I)$, then $S \subseteq T$ for some $T \in \mcC$. If, in addition, $S$ is a maximal subring of $R$, then $S \in \mcC$.
\end{lem}
\begin{proof}
Suppose that $\sigma(R) < \sigma(R/I)$, but that $S \not\subseteq T$ for any $T \in \mcC$. Then, the collection
\begin{equation*}
\mcC_S \colonequals \{T \cap S : T \in \mcC \}
\end{equation*}
is a proper cover of $S$, which implies that
\begin{equation*}
\sigma(R/I) = \sigma(S) \le |\mcC| = \sigma(R) < \sigma(R/I),
\end{equation*}
a contradiction. Thus, $S \subseteq T$ for some $T \in \mcC$, and if $S$ is maximal then $S = T$.
\end{proof}

\section{Reduction Theorems}\label{Reduction section}

Here, we prove that if $R$ is a ring with unity such that its covering number $\sigma(R)$ is finite, then there exists a residue ring $\olR$ of $R$ such that $\sigma(R) = \sigma(\olR)$ and $\olR$ is finite of characteristic $p$ for some prime $p$, and the Jacobson radical of $\olR$ has nilpotency 2. Thus, the determination of finite covering numbers can always be reduced to a class of finite dimensional algebras over $\F_p$.

\begin{prop}\label{reduction to char p^n}
Let $R$ be a ring such that $\sigma(R)$ is finite. Then, there exists a two-sided ideal $I$ of $R$ such that $\sigma(R/I) = \sigma(R)$, $R/I$ is finite, and $R/I$ has characteristic $p^n$ for some prime $p$ and some $n \ge 1$.
\end{prop}
\begin{proof}
The fact that a two-sided ideal $I_0$ exists with $R/I_0$ finite and $\sigma(R/I_0) = \sigma(R)$ follows from the work of Neumann \cite[Lem.\ 4.1, 4.4]{Neumann} and Lewin \cite[Lem.\ 1]{Lewin}. Next, it is well known \cite[Thm.\ I.1]{McD} that any finite ring with identity is isomorphic to a direct product of rings of prime power order. Assume that $R/I_0 = \prod_{i=1}^t R_i$, where each $R_i$ is a ring with unity and $|R_i|=p_i^{n_i}$ for some distinct primes $p_1, \ldots, p_t$ and some positive integers $n_1, \ldots, n_t$. Then, by \cite[Cor.\ 2.4]{Werner}, $\sigma(R/I_0) = \min_{1 \le i \le t}\{\sigma(R_i)\}$.
\end{proof}

Hence, to determine the covering numbers of rings, we may restrict out attention to rings of order $p^n$. In fact, we can do better and assume that $n=1$.

\begin{lem}\label{pR in M lem}
Let $R$ be a finite ring of characteristic $p^n$. 
\begin{enumerate}[(1)]
\item Let $S$ be a subring of $R$ such that $R = S + pR$. Then, $S=R$.
\item Let $M$ be a maximal subring of $R$. Then, $pR \subseteq M$.
\end{enumerate}
\end{lem}
\begin{proof}
For (1), let $r \in R$. Then, there exist $s_0 \in S$ and $r_0 \in R$ such that $r = s_0 + p r_0$. Similarly, there exist $s_1 \in S$ and $r_1 \in R$ such that $r_0 = s_1 + p r_1$. So, $r = s_0 + ps_1 + p^2r_1$. Continuing in this manner, we may write
\begin{equation*}
r = s_0 + ps_1 + p^2 s_2 + \cdots + p^{n-1}s_{n-1}
\end{equation*}
for some $s_0, \ldots, s_{n-1} \in S$. Thus, $r \in S$ and $R=S$.

For (2), note that $M+pR$ is a subring of $R$; closure under addition is clear, and closure under multiplication follows from the fact that $p$ is central in $R$. So, we have
\begin{equation*}
M \subseteq M+pR \subseteq R.
\end{equation*}
By maximality, either $M+pR=M$ or $M+pR=R$. In the latter case, $M=R$ by part (1), a contradiction. Thus, $M=M+pR$ and $pR \subseteq M$.
\end{proof}

\begin{prop}\label{char p prop}
Let $R$ be a finite ring of characteristic $p^n$. Then, $\sigma(R) = \sigma(R/pR)$.
\end{prop}
\begin{proof}
We know that $\sigma(R) \le \sigma(R/pR)$. For the other inequality, assume that $R = \bigcup_{i=1}^t M_i$ is a minimal cover of $R$ by maximal subrings $M_1, \ldots, M_t$. By Lemma \ref{pR in M lem}, $pR \subseteq M_i$ for each $1 \le i \le t$, so $M_i / pR$ is a proper subring of $R/pR$ for each $i$. Thus, $R/pR = \bigcup_{i=1}^t M_i/pR$ and $\sigma(R/pR) \le t = \sigma(R)$.
\end{proof}

By combining Propositions \ref{reduction to char p^n} and \ref{char p prop}, we obtain the promised theorem.

\begin{thm}\label{reduction to char p}
Let $R$ be a ring such that $\sigma(R)$ is finite. Then, there exists a two-sided ideal $I$ of $R$ such that $\sigma(R/I) = \sigma(R)$, $R/I$ is finite, and $R/I$ has characteristic $p$ for some prime $p$.
\end{thm}

From here, we assume that $R$ is a finite ring of characteristic $p$. It remains to demonstrate that we may assume $\msJ(R)^2= \{0\}$. To do this, we recall two results: Nakayama's Lemma \cite[4.22]{Lam}, and the Wedderburn-Malcev Theorem \cite[Sec.\ 11.6, Cor.\ p.\ 211]{Pierce}, \cite[Thm.\ VIII.28]{McD} (also known as the Wedderburn Principal Theorem).

\begin{lem}\label{Nakayama lem} (Nakayama's Lemma)
Let $A$ be a ring with unity. Let $I \subseteq A$ be a left (resp.\ right) ideal of $A$. Then, $I \subseteq \msJ(A)$ if and only if for any left (resp.\ right) $A$-modules $N \subseteq M$ such that $M/N$ is finitely generated, $N + I \cdot M = M$ (resp.\ $N + M \cdot I = M$) implies that $N = M$.
\end{lem}

\begin{thm}\label{Wedderburn thm} (Wedderburn-Malcev Theorem)
Let $R$ be a finite ring with unity of characteristic $p$. Then, there exists a subalgebra $S$ of $R$ such that $R = S \oplus \msJ(R)$, and $S \cong R/\msJ(R)$ as algebras. Moreover, $S$ is unique up to conjugation by elements of $1+\msJ(R)$.
\end{thm}

In light of the Wedderburn-Malcev Theorem, we adopt the following notation.

\begin{Not}\label{S and J notation}
Let $\msS(R)$ be the set of all semisimple subalgebras of $R$ that satisfy the conclusion of the Wedderburn-Malcev Theorem. That is, 
\begin{equation*}
\msS(R) = \{ S \subseteq R : S \text{ is a subalgebra of } R, S \cong R/\msJ(R), \text{ and } R = S \oplus \msJ(R)\}.
\end{equation*}
\end{Not}

We shall prove below in Theorem \ref{Max subring classification} that when $R$ has characteristic $p$, the maximal subrings of $R$ fall into two classes: those containing $\msJ(R)$, and those of the form $S \oplus I$, where $S \in \msS(R)$ and $I$ is a maximal $R$-subideal of $\msJ(R)$. To reach this conclusion, we require a few more technical lemmas.


\begin{lem}\label{lem2}
Let $J = \msJ(R)$. Let $S$ be a subring of $R$ such that $S/(S \cap J) \cong R/J$. Then, $S \cap J = \msJ(S)$.
\end{lem}
\begin{proof}
On the one hand, $S \cap J$ is a nilpotent ideal of $S$ because it is contained in $J$. So, $S \cap J \subseteq \msJ(S)$ by \cite[Prop.\ IV.7]{McD}. On the other hand, $S/(S \cap J)$ is semisimple because it is isomorphic to $R/J$. Hence, $S \cap J \supseteq \msJ(S)$ by \cite[Prop.\ IV.8]{McD}.
\end{proof}

\begin{prop}\label{S cap J prop}
Let $J = \msJ(R)$, and let $T$ be a maximal subring of $R$. Then, $T \cap J$ is a two-sided ideal of $R$.
\end{prop}
\begin{proof}
The result is trivial if $J \subseteq T$, so assume that $J \not\subseteq T$. Let $\mcJ$ denote the two-sided ideal of $R$ generated by $T \cap J$. Certainly, $T \cap J \subseteq \mcJ$; and, since $J$ is a two-sided ideal of $R$, we have $\mcJ \subseteq J$. So, it suffices to prove that $\mcJ \subseteq T$.

Apply Nakayama's Lemma with $A=T$, $I=T\cap J$, $M=R$, and $N=T$. Since $T$ is maximal and $J \not\subseteq T$, we have
\begin{equation*}
T/(T \cap J) \cong (T+J)/J = R/J
\end{equation*}
so $T \cap J = \msJ(T)$ by Lemma \ref{lem2}. Next, note that $T + \msJ(T)R$ is a subring of $R$. By the maximality of $T$, either $T+\msJ(T)R = T$ or $T+\msJ(T)R=R$. If $T+\msJ(T)R=R$, then $T=R$ by Nakayama's Lemma, which is a contradiction. Thus, $T + \msJ(T)R = T$, implying that $\mcJ=\msJ(T)R \subseteq T$, which completes the proof.
\end{proof}

\begin{thm}\label{Max subring classification}
Let $T$ be a maximal subring of $R$. Let $J = \msJ(R)$.
\begin{enumerate}[(1)]
\item $J \subseteq T$ if and only if $T$ is the inverse image of a maximal subring of $R/J$.

\item $J \not\subseteq T$ if and only if $T = S \oplus \msJ(T)$, where $S \in \msS(R)$ and $\msJ(T) = T \cap J$ is an ideal of $R$ that is maximal among the subideals of $R$ contained in $J$.
\end{enumerate}
\end{thm}
\begin{proof}
Part (1) is clear. For part (2), first let $S \in \msS(R)$ and let $I \subseteq J$ be an ideal of $R$ that is maximal among the subideals of $J$. Then, $S \oplus I$ is a maximal subring of $R$, and $\msJ(S \oplus I) = I$ because $S \cap J = \{0\}$.

Conversely, assume that $J \not\subseteq T$. Then, $T$ is not a maximal ideal of $R$, so $1 \in T$ by Lemma \ref{1 in max lem}. Hence, we may apply Theorem \ref{Wedderburn thm} to $T$, which gives $T = S \oplus \msJ(T)$ for some $S \in \msS(T)$. Next, since $T$ is maximal, we have $T/(T \cap J) \cong R/J$, and so $\msJ(T) = T \cap J$ by Lemma \ref{lem2}. It follows that $S \in \msS(R)$, because $S \cong T/\msJ(T) \cong R/J$. Finally, $\msJ(T) \subseteq J$ is a two-sided ideal of $R$ by Proposition \ref{S cap J prop}. By assumption, $\msJ(T) \ne J$, and the maximality of $T$ implies that $\msJ(T)$ is maximal among the subideals of $R$ contained in $J$.
\end{proof}

\begin{cor}\label{J2 cor}
Let $J = \msJ(R)$. Then, any maximal subring $T$ of $R$ contains $J^2$. Consequently, $\sigma(R) = \sigma(R/J^2)$.
\end{cor}
\begin{proof}
Let $T$ be a maximal subring of $R$. If $J \subseteq T$, then certainly $J^2 \subseteq T$. If $J \not\subseteq T$, then by Theorem \ref{Max subring classification}, $\msJ(T)$ is maximal among the subideals of $R$ contained in $J$. Suppose that $J^2 \not\subseteq \msJ(T)$. Then, $\msJ(T) + J^2 = J$ by the maximality of $\msJ(T)$. Applying Nakayama's Lemma with $N=\msJ(T)$, $I=J$, and $M=J$ yields $\msJ(T)=J$, which is a contradiction. So, $J^2 \subseteq \msJ(T) \subseteq T$.

For the equality of the covering numbers, we always have $\sigma(R) \le \sigma(R/J^2)$. But, given a minimal cover $\mcC$ of $R$ by maximal subrings, $J^2 \subseteq T$ for all $T \in \mcC$. Hence, $T/J^2$ is a proper subring of $R/J^2$, and so $\{T/J^2 : T \in \mcC\}$ covers $R/J^2$. Thus, $\sigma(R/J^2) \le \sigma(R)$.
\end{proof}

Summarizing all the results of this section yields the full reduction theorem on rings with a finite covering number.

\begin{thm}\label{full reduction thm}
Let $R$ be a (unital, associative) ring such that $\sigma(R)$ is finite. Then, there exists a two-sided ideal $I$ of $R$ such that $R/I$ is finite; $R/I$ has characteristic $p$; $\msJ(R/I)^2 = \{0\}$; and $\sigma(R/I) = \sigma(R)$.
\end{thm}

\section{Covering Numbers of Commutative Rings}\label{Commutative section}

Throughout this section, $R$ is a finite commutative ring of characteristic $p$, $J = \msJ(R)$, and $J^2=\{0\}$. From the Wedderburn-Malcev Theorem, we know that there exists a semisimple subring $S \subseteq R$ such that $S \cong R/J$ and $R = S \oplus J$. Furthermore, $S$ is unique up to conjugation by the units in $1 + J$. However, since $R$ is commutative, this conjugation action is trivial. We conclude that there is a \textit{unique} semisimple subring $S$ of $R$ such that $R = S \oplus J$.

The decomposition $R=S \oplus J$ suggests two methods for covering $R$ by subrings. First, given a cover $S = \bigcup_{i=1}^t S_i$ of $S$ by maximal subrings $S_i$ of $S$, we have $R = \bigcup _{i=1}^t (S_i \oplus J)$ (this is equivalent to lifting a cover of $R/J$ up to $R$). Second, one could attempt to cover $J$ by maximal subideals of $R$. If $J = \bigcup_{j=1}^m I_j$, where each $I_j$ is an ideal of $R$, then $S \oplus I_j$ is a subring of $R$, and $R = \bigcup_{j=1}^m (S \oplus I_j)$. We proceed to show that, if $R$ is coverable, then a minimal cover can always be produced by one of these methods.

Let us first establish some notation and consider when and how $J$ could be covered by subideals. Since $S$ is commutative and semisimple, it is a direct sum of finite fields: $S = F_1 \oplus \cdots \oplus F_n$, where each $F_i$ is a field. Since $R$ is unital, we have $S \cdot J \ne \{0\}$ if $J \ne \{0\}$. However, it is possible that $F_i \cdot J =\{0\}$ for some $1 \le i \le n$.

\begin{Def}\label{def:Lambda}
Let $S = F_1 \oplus \cdots \oplus F_n$ as above. For each $1 \le i \le n$, let $e_i = 1_{F_i}$, so that $1_R = e_1 + \cdots + e_n$ and $\{e_1, \ldots, e_n\}$ is a set of primitive orthogonal idempotents. Define $\Lambda \colonequals \{1 \le i \le n : e_i J \ne \{0\}\}$. Then, $e_i J$ is a nonzero $F_i$-vector space for each $i \in \Lambda$. For every $i \in \Lambda$, define $d_i \colonequals \dim_{F_i} e_i J$. Finally, let $\Lambda_2 \colonequals \{i \in \Lambda : d_i \ge 2\}$.

We say that $J$ is \textit{coverable} if $J = \bigcup_{j=1}^m I_j$ for some proper $R$-subideals $I_j \subseteq J$. If such a union is possible, then we let $\sigma(J)$ be the size of a minimal cover. Likewise, we say that $e_i J$ is \textit{coverable} if $e_i J = \bigcup_{j=1}^m V_j$ for some proper $F_i$-submodules $V_j \subseteq e_i J$; and $\sigma(e_i J)$ is the size of a minimal cover (if one exists).
\end{Def}

Since $e_i J$ is an $F_i$-vector space, computing $\sigma(e_i J)$ is equivalent to finding a minimal cover of a finite dimensional vector space by subspaces. The existence of such covers, as well as the relevant covering numbers, are well known in these cases (for instance, see \cite{Khare}).

\begin{prop}\label{Vector space coverings}
With notation as above, $e_i J$ is coverable if and only if $d_i \ge 2$. If $e_i J$ is coverable, then $\sigma(e_i J) = |F_i|+1$.
\end{prop}

As we now demonstrate, $J = \bigoplus_{i=1}^n e_i J$ is a decomposition of $J$ into $R$-subideals, and every subideal of $J$ respects this decomposition. Consequently, questions about covering $J$ reduce to the analogous problems for $e_i J$.

\begin{lem}\label{Subideal decomposition}
Let $I \subseteq J$ be an ideal of $R$. 
\begin{enumerate}[(1)]
\item $e_i I$ is an ideal of $R$ for each $1 \le i \le n$, and $I = \bigoplus_{i=1}^n e_i I$.
\item There exists an ideal $\overline{I}$ of $R$ such that $J = I \oplus \overline{I}$.
\end{enumerate}
\end{lem}
\begin{proof}
(1) For each $i$, $e_i I$ is an ideal of $R$ because $R$ is commutative.  Consider $\alpha \in e_i I \cap e_j I$ with $i \ne j$. We have $\alpha = e_i x = e_j y$ for some $x, y \in I$. But, since $e_i$ and $e_j$ are orthogonal, \begin{equation*}
\alpha = e_i x = e_i(e_i x) = e_i(e_j y) = 0.
\end{equation*}
Thus, $e_i I \cap e_j I = \{0\}$ and $I = \bigoplus_{i=1}^n e_i I$.

(2) Since $S$ is semisimple, both $J$ and $I$ are semisimple left $S$-modules. So, there exists a left $S$-module $\overline{I} \subseteq J$ such that $J = I \oplus \overline{I}$ (as left $S$-modules). We claim that $\overline{I}$ is an ideal of $R$. Let $r \in R$ and $\alpha \in \overline{I}$, and write $r = s + x$ for some $s \in S$ and $x \in J$. Since $J^2=\{0\}$, we have $r\alpha = s\alpha \in \overline{I}$. Thus, $\overline{I}$ is a left $R$-module; but, $R$ is commutative, so $\overline{I}$ is a two-sided ideal of $R$, as desired.
\end{proof}

\begin{thm}\label{When J is coverable}
The following are equivalent.
\begin{enumerate}[(1)]
\item $J$ is coverable by $R$-subideals.
\item At least one $e_iJ$ is coverable by $F_i$-submodules.
\item $\Lambda_2 \ne \varnothing$.
\end{enumerate}
If any of these conditions hold, then $\sigma(J) = \min_{i \in \Lambda_2}(|F_i|) + 1$.
\end{thm}
\begin{proof}
(3) $\Rightarrow$ (2) holds by Proposition \ref{Vector space coverings}. For (2) $\Rightarrow$ (1), assume that $e_i J = \bigcup_{j=1}^m V_j$ is a covering of $e_i J$. Then, $e_i V_j = V_j$ for each $1 \le j \le m$, and 
\begin{equation*}
I_j\colonequals\big( \bigoplus_{k \ne i} e_k J \big) \oplus e_i V_j
\end{equation*}
is a proper $R$-subideal of $J$. We then have $J = \bigcup_{j=1}^m I_j$, and $J$ is coverable.

Next, for (1) $\Rightarrow$ (3), suppose that $\Lambda_2 = \varnothing$. So, $d_i = 1$ for all $i \in \Lambda$. For each $i \in \Lambda$, pick a nonzero $x_i \in e_i J$, and let $x=\sum_{i \in \Lambda} x_i$. We claim that $x$ generates $J$. Indeed, since $d_i=1$, $e_i J = F_i x_i$. So, given $y \in J$, we have $y = \sum_{i \in \Lambda} e_i y$, and for each $i$ there exists $a_i \in F_i$ such that $e_i y = a_i x_i$. Let $r = \sum_{i \in \Lambda} a_i$. Then,
\begin{equation*}
y = \sum_{i \in \Lambda} a_i x_i = \Big(\sum_{i \in \Lambda} a_i\Big)\cdot \Big(\sum_{i \in \Lambda} x_i\Big) = rx.
\end{equation*}
Hence, $J = Rx$, and $J$ cannot be covered by subideals since no proper subideal of $J$ can contain $x$.

The final claim of the theorem follows from a variation of \cite[Thm.\ 2.2]{Werner}. That theorem is stated for rings with unity, but the proof also works in the current setting. For completeness, we include the adapted proof below.

Assume that $J$ is coverable. We use induction on $|\Lambda|$. There is nothing to prove if $|\Lambda|=1$, so assume that $|\Lambda|>1$ . The proof of (2) $\Rightarrow$ (1) above shows that $\sigma(R) \le \sigma(e_i J)$ whenever $e_i J$ is coverable, so $\sigma(J) \le \min_{i \in \Lambda_2}(|F_i|) + 1$. For the reverse inequality, assume without loss of generality that $d_1 \ge 2$ and that $m \colonequals \min_{i \in \Lambda_2}(|F_i|) + 1 = \sigma(e_1 J)$. Let $t < m$ and let $K_1, \ldots, K_t$ be maximal $R$-subideals of $J$. We will show that $\bigcup_{j=1}^t K_j \ne J$.

Let $J' = \bigoplus_{i \in \Lambda, i > 1} e_i J$. By induction, either $J'$ is not coverable, or $\sigma(J') \le m$. By Lemma \ref{Subideal decomposition}, for each $1 \le j \le m$ there exist subideals $L_j \subseteq e_1 J$ and $M_j \subseteq J'$ such that $K_j = L_j \oplus M_j$; we admit the possibility that some $L_j = e_1 J$, or some $M_j = J'$. Clearly, $\bigcup_{j=1}^t K_j \subseteq (\bigcup_{j=1}^t L_j) \oplus (\bigcup_{j=1}^t M_j)$, so there is nothing to prove if either $\bigcup_{j=1}^t L_j \ne e_i J$ or $\bigcup_{j=1}^t M_j \ne J'$. 

From here, assume that $\bigcup_{j=1}^t L_j = e_i J$ and $\bigcup_{j=1}^t M_j = J'$. Since $t < m$, the union of all the $L_j$ properly contained in $e_1 J$ cannot equal $e_1 J$; likewise, the union of the proper $M_j$ cannot equal $J'$. Choose $x \in e_1 J$ that is not in the union of the proper $L_j$, and choose $y \in J'$ similarly.

If $x+y \in K_\ell$ for some $1 \le \ell \le t$, then $x \in L_\ell$ and $y \in M_\ell$. Consequently, $L_\ell = e_1 J$ and $M_j = J'$; but then, $K_\ell = J$, a contradiction. Thus, $\bigcup_{j=1}^t K_j \ne J$, and therefore $\sigma(J) = m$.
\end{proof}

We now examine the relationship between coverings of $R$ and coverings of $J$.

\begin{thm}\label{R and J commutative thm}
Assume that $R$ is coverable.
\begin{enumerate}[(1)]
\item If $J$ is not coverable, then $\sigma(R) = \sigma(R/J)$.
\item If $J$ is coverable and $R$ is $\sigma$-elementary, then $\sigma(R)$ is equal to the number of maximal $R$-subideals of $J$.
\end{enumerate}
\end{thm}
\begin{proof}
(1) Assume that $J$ is not coverable. Then, there exists $x \in J$ such that $Rx = J$. Since $J^2 = \{0\}$, this means that $Sx = J$.

Next, by Theorem \ref{Max subring classification}, any maximal subring of $R$ either contains $J$, or has the form $S \oplus I$ for some maximal subideal $I$ of $J$ (note that $\msS(R) = \{S\}$ because $R$ is commutative). Let $\mcC = \{T_1, \ldots, T_k\} \cup \{A_1, \ldots, A_m\}$ be a minimal cover of $R$, where $J \not\subseteq T_i$ for all $1 \le i \le k$, and $J \subseteq A_j$ for all $1 \le j \le m$.

Let $a \in S$. If $a + x \in T_i$ for some $i$, then $x \in T_i$ because $S \subseteq T_i$. But then, $Sx = J \subseteq T_i$, a contradiction. So, for all $a \in S$, $a + x \in \bigcup_{j=1}^m A_j$; hence, $S+x \subseteq \bigcup_{j=1}^m A_j$. This means that $\{A_1/J, \ldots, A_m/J\}$ forms a cover of $R/J \cong S$. Thus, $\sigma(R/J) \le m \le \sigma(R)$. Since we always have $\sigma(R) \le \sigma(R/J)$, we conclude that $\sigma(R) = \sigma(R/J)$.

(2) Assume that $J$ is coverable and that $R$ is $\sigma$-elementary. Let $\mcC$ be a minimal cover of $R$, and let $m$ be the number of maximal $R$-subideals of $J$. Let $\msT$ be the set of all maximal subrings of $R$ that do not contain $J$. By Theorem \ref{Max subring classification} and the uniqueness of $S$, if $T \in \msT$, then $T = S \oplus I$ for some maximal subideal $I$ of $J$; conversely, $S \oplus I \in \msT$ for any such $I$. Thus, $|\msT| = m$. 

Next, by Lemma \ref{Subideal decomposition}, for every $T  = S \oplus I \in \msT$,  $I$ has an ideal complement $\overline{I}$ in $J$. So, $R = T \oplus \overline{I}$, and by Lemma \ref{Sigma elementary lemma}, $T \in \mcC$. Thus, $\sigma(R) = |\mcC| \ge m$.

To complete the proof, it suffices to show that $\bigcup_{T \in \msT} T = R$. So, let $s + x \in R$, where $s \in S$ and $x \in J$. Since $J$ is coverable, $x$ is an element of some maximal subideal $I \subseteq J$, and $s + x \in S \oplus I \in \msT$. Thus, $\bigcup_{T \in \msT} T$ is equal to $R$, and $\sigma(R)=m$.
\end{proof}

All that remains is to count the number of maximal subideals of $J$.

\begin{thm}\label{Counting subideals of J}
Let $m$ be the number of maximal subideals of $J$. Then, $m=\sum_{i \in \Lambda} (|F_i|^{d_i}-1)/(|F_i|-1)$. If $J$ is coverable and $R$ is $\sigma$-elementary, then $\sigma(R) = m =  \min_{i \in \Lambda_2}(|F_i| + 1)$. Moreover, in this case, $\Lambda = \Lambda_2$, $|\Lambda_2|=1$, and $d_i=2$ for $i \in \Lambda_2$.
\end{thm}
\begin{proof}
Let $I$ be a maximal subideal of $J$. By Lemma \ref{Subideal decomposition}, $I$ decomposes as $I = \bigoplus_{i \in \Lambda} e_i I$, and each $e_i I$ is an ideal of $R$. Because of this, $I$ is maximal if and only if $e_k I$ is maximal in $e_k J$ for a single $k \in \Lambda$ and $e_j I = e_j J$ for all $j \ne k$. Next, since each $e_i J$ is an $F_i$-vector space of dimension $d_i$, the maximal subideals in $e_i J$ are either the zero space (if $d_i=1$), or are the subspaces of codimension 1 (if $d_i \ge 2$). In either case, the number of maximal subspaces is equal to $(|F_i|^{d_i}-1)/(|F_i|-1)$. It follows that $m=\sum_{i \in \Lambda} (|F_i|^{d_i}-1)/(|F_i|-1)$.

Now, assume that $J$ is coverable and $R$ is $\sigma$-elementary. By Theorem \ref{When J is coverable}, $|\Lambda_2|\ge 1$. So, on the one hand, 
\begin{equation*}
\min_{i \in \Lambda_2}(|F_i| + 1) \le \sum_{i \in \Lambda} \dfrac{|F_i|^{d_i}-1}{|F_i|-1} = m.
\end{equation*}
On the other hand, by Theorem \ref{R and J commutative thm}, $m = \sigma(R)$. Moreover, $\sigma(R) \le \sigma(J)$, and $\sigma(J) = \min_{i \in \Lambda_2}(|F_i| + 1)$. Thus, $\sigma(R) = m = \min_{i \in \Lambda_2}(|F_i| + 1)$. Finally, in order for $\min_{i \in \Lambda_2}(|F_i| + 1)$ to equal $\sum_{i \in \Lambda} (|F_i|^{d_i}-1)/(|F_i|-1)$, we must have $\Lambda = \Lambda_2$, $|\Lambda_2|=1$, and---for the unique $i \in \Lambda_2$---$|F_i| + 1 = (|F_i|^{d_i}-1)/(|F_i|-1)$, which means that $d_i=2$.
\end{proof}

At this point, we have everything we need to classify commutative $\sigma$-elementary rings. The classification is divided into two cases, depending on whether or not $J = \{0\}$. When $J = \{0\}$, $R$ is a direct product of finite fields, and the classification can be determined by several results from \cite{Werner}, which we summarize here.

\begin{thm}\label{thm:Werner1}\mbox{}
\begin{enumerate}[(1)]
\item \cite[Thm.\ 5.1]{Werner} Let $R = \prod_{i=1}^t ( \prod_{j=1}^{t_i} F_i)$, where each $F_i$ is a distinct finite field. Then, $R$ is coverable if and only if at least one $\prod_{j=1}^{t_i} F_i$ is coverable. If $R$ is coverable, then $\sigma(R) = \min_{i} \sigma(\prod_{j=1}^{t_i} F_i)$.

\item \cite[Thm.\ 3.5, 5.4]{Werner} For each prime $p$ and each $n \ge 1$, let $\psi(p,n)$ be the number of monic irreducible polynomials in $\F_p[x]$ of degree $n$. For a prime power $q=p^n$, define
\begin{equation*}
\tau(q) = \begin{cases} p & n=1 \\ \psi(p,n)+1 & n > 1 \end{cases}
\end{equation*}
Then, $\prod_{i=1}^t \F_q$ is coverable if and only if $t \ge \tau(q)$. Moreover, if $t \ge \tau(q)$, then $\sigma(\prod_{i=1}^t \F_q) = \sigma(\prod_{i=1}^{\tau(q)} \F_q)$.
\end{enumerate}
\end{thm}

\begin{thm}\label{thm:comm_sigma}
Let $R$ be a commutative, $\sigma$-elementary ring with unity for which $\sigma(R)$ is finite, and let $J = \msJ(R)$.
\begin{enumerate}[(1)]
\item If $J = \{0\}$, then $R \cong \prod_{i=1}^{\tau(q)} \F_q$ for some finite field $\F_q$.

\item If $J \ne \{0\}$, then for some finite field $\F_q$, $R$ is isomorphic to the following ring of $3 \times 3$ matrices:
\begin{equation*}
\left\{ \begin{pmatrix} a & b & c \\ 0 & a & 0 \\ 0 & 0 & a \end{pmatrix} : a, b, c \in \F_q \right\}.
\end{equation*}
\end{enumerate}
\end{thm}
\begin{proof}
By Theorem \ref{full reduction thm}, we may assume that $R$ is finite, has characteristic $p$, and that $J^2=\{0\}$. 

(1) Assume first that $J = \{0\}$. Then, $R$ semisimple, and since $R$ is commutative, we have $R = \prod_{i=1}^t ( \prod_{j=1}^{t_i} F_i)$, where each $F_i$ is a distinct finite field of characteristic $p$. By Theorem \ref{thm:Werner1}, there is a field $\F_q$ for which $\sigma(R) = \sigma(\prod_{i=1}^{\tau(q)} \F_q)$. Let $R' = \prod_{i=1}^{\tau(q)} \F_q$. Then, $R'$ is $\sigma$-elementary, and occurs as a residue ring of $R$. Hence, we must have $R = R'$.

(2) Assume that $J \ne \{0\}$ and let $S$ be the semisimple subalgebra of $R$ such that $S \cong R/J$. As in Definition \ref{def:Lambda}, let $S = F_1 \oplus \cdots \oplus F_n$ for some finite fields $F_i$. Let $\Lambda_0 = \{1 \le i \le n : F_i \cdot J = \{0\}\}$, and let $S_0 = \bigoplus_{i \in \Lambda_0} F_i$. Then,
\begin{equation*}
S = S_0 \oplus \Big( \bigoplus_{i \in \Lambda} F_i \Big),
\end{equation*}
where $\Lambda$ is defined as in Definition \ref{def:Lambda}.  Since $R$ is $\sigma$-elementary, we have $\sigma(R) < \sigma(R/J)$. So, $J$ is coverable by Theorem \ref{R and J commutative thm}, and by Theorem \ref{Counting subideals of J}, $\Lambda = \Lambda_2$, $|\Lambda_2|=1$, and $J$ is 2-dimensional. This means that there is a unique field $\F_q$ such that $S = S_0 \oplus \F_q$ and $\dim_{\F_q} J = 2$ 

It suffices to show that $S_0=\{0\}$, since then $R = \F_q \oplus J$ with $\dim_{\F_q} J = 2$ and $J^2=\{0\}$, and these properties characterize the matrix ring given in the statement of the theorem. 

Since $S_0 \cdot J = \{0\}$, $S_0$ is an ideal of $R$, and $R/S_0 \cong \F_q \oplus J$. So, $\sigma(R) \le \sigma(\F_q \oplus J)$. But, by Theorem \ref{R and J commutative thm}, a minimal cover of $R$ consists of subrings of the form $S \oplus I$, where $I$ is a maximal subideal of $J$. Contracting each of these subrings to $\F_q \oplus J$ forms a cover of that ring, so $\sigma(\F_q \oplus J) \le \sigma(R)$. Because $R$ is $\sigma$-elementary, this forces $S_0 = \{0\}$ and $R = \F_q \oplus J$, as required.
\end{proof}

\begin{cor}\label{Full commutative case}
Let $R$ be a commutative ring with unity such that $\sigma(R)$ is finite. Then, either $\sigma(R) = \sigma(R/\msJ(R))$, or $\sigma(R) = p^d+1$ for some positive integer $d$.
\end{cor}
\begin{proof}
Since $R$ is commutative and coverable, it has a $\sigma$-elementary residue ring $\olR$ such that $\sigma(R) = \sigma(\olR)$. Applying Theorem \ref{thm:comm_sigma} to $\olR$ (and Theorem \ref{Counting subideals of J} to $\olR$, if $\olR$ has nontrivial Jacobson radical) yields the desired result.
\end{proof}
\section{\texorpdfstring{Finite Rings with $R/\msJ(R)$ Commutative}{Finite Rings with R/J(R) Commutative}}

In the previous section, we proved that if $R$ is finite, commutative, coverable ring, then either $\sigma(R)=\sigma(R/\msJ(R))$, or $\sigma(R)=p^d+1$ for some prime power $p^d$. It turns out that this characterization of $\sigma(R)$ holds with the weaker assumption that $R/\msJ(R)$ is commutative. Thus, it can be applied to rings such as upper triangular matrix rings, or certain group algebras such as $\F_2[D_8]$ (which has $R/\msJ(R) \cong \F_2$) or $\F_2[A_4]$ (which has $R/\msJ(R) \cong \F_2 \times \F_4$). Proving the characterization, however, is more complicated because $R$ may be noncommutative in these case.

We will maintain some notation and assumptions from prior sections. Throughout this section, $R$ is a finite ring with unity in characteristic $p$, $J = \msJ(R)$, and $J^2 =\{0\}$. Recall that
\begin{equation*}
\msS(R) = \{ S \subseteq R : S \text{ is a subalgebra of } R, S \cong R/\msJ(R), \text{ and } R = S \oplus \msJ(R)\}.
\end{equation*}
Since $R$ need not be commutative, there need not be a unique $S \in \msS(R)$ such that $R = S \oplus J$. As we shall see, our arguments are ultimately divided into two cases, depending on whether or not there exists a maximal subring $T \subseteq R$ such that $\bigcup_{S \in \msS(R)} S \subseteq T$. Both cases may occur in practice. For instance, such a $T$ exists for $\F_2[D_8]$, but not for $\F_2[A_4]$ or an upper triangular matrix ring.

First, we establish a number of properties of $J$ and the elements of $\msS(R)$. Most of these are consequences of our assumption that $J^2 = \{0\}$, and do not require $S$ to be commutative. Later, we will specialize to the case where $S$ is commutative and relate things back to covers of $R$.

The first lemma collects the basic arithmetic properties of $J$ and $1 + J$. These will be used freely in the remainder of this section. In what follows, when $r \in R$ and $u \in R^\times$, we let $r^u \colonequals u^{-1} r u$.

\begin{lem}\label{lem:1+Jarithmetic}
 Let $x, y \in J$ and $n \in \N$.  
 \begin{itemize}
 \item[(1)] $xy = 0$.
 \item[(2)] $(1 + x)(1 + y) = 1 + x + y$.
 \item[(3)] $(1 + x)^n = 1 + nx$, $(1 + x)^{-1} = 1 - x$, and the order of $(1 + x)$ is $p$.
 \item[(4)] $x \cdot (1 + y) = x$.
 \item[(5)] $1 + J$ is an elementary abelian $p$-group.
\end{itemize}
\end{lem}

\begin{proof}
 First, $xy = 0$ since $J^2 = \{0\}$.  Furthermore, $(1 + x)(1 + y) = 1 + x + y$ since $xy = 0$ by (1), proving (2).  Now, for all $n \in \N$, $(1 + x)^n = 1 + nx$, so $(1 + x)^{-1} = (1+x)^{p-1} = 1 - x$, proving (3).  Next,
\begin{equation*}
x \cdot (1 + y) = x + xy = x,
\end{equation*}
proving (4). Lastly, since $(1 + x)(1 + y) = 1 + x + y = 1 + y + x$, the group $1 + J$ is abelian, and this along with (3) proves (5).
\end{proof}

\begin{lem}\label{lem:conj1+J}
Let $S \in \msS(R)$ and $s \in S$.  If $x \in J$, then 
\begin{equation*}
s^{1 + x} = s + (sx - xs).
\end{equation*}
In particular, $s^{1 + J} \subseteq s + J$.
\end{lem}

\begin{proof}
Let $s \in S$ and $x \in J$.  Then,
\begin{equation*}
s^{1 + x} = (1 + x)^{-1}\cdot s \cdot (1 + x) = (1 - x)(s + sx) = s + (sx - xs),
\end{equation*}
since $J^2 = \{0\}$.  Moreover, since $J$ is a two-sided ideal, $s^{1 + x} \in s + J$ for all $x \in J$. 
\end{proof}

Recall by Theorem \ref{Max subring classification} that if $S \in \msS(R)$ and $T$ is a maximal subring containing $S$, then $T = S \oplus I$, where $I = \msJ(T) = T \cap J$ is maximal among the subideals of $R$ contained in $J$.


\begin{lem}\label{S and J lem}\mbox{}
\begin{enumerate}[(1)]
\item Let $S, S' \in \msS(R)$.  If there is a maximal subring $T = S \oplus I$ such that $S, S' \subseteq T$, then $S' = S^{1 + x}$ for some $x \in I$.  

\item Let $S \in \msS(R)$ and $T = S \oplus I$ be a maximal subring containing $S$.  Then every conjugate of $T$ in $R$ is of the form $T^{1+x}$, where $x \in J$, and $T^{1 + x} = S^{1 + x} \oplus I$. Consequently, there are at most $|J:I|$ conjugates of $T$ in $R$.

\item Assume $x \in J$ and $S^{1 + x} = S$.  Then, $x \in Z(R)$. Moreover, if $s \in Z(S)$, then $S^{1 + sx} = S$ for all $s \in S$.
\end{enumerate}
\end{lem}
\begin{proof}
(1) Apply the Wedderburn-Malcev Theorem (Theorem \ref{Wedderburn thm}) to $T = S \oplus I$.

(2) By assumption, $T$ is a maximal subring of $R$ not containing $J$, and $I$ is a maximal subideal of $J$. Let $s \in S, y \in I$, and $u \in R^\times$. Then, $(s+y)^u = s^u + y^u \in S^u + I \subseteq T^u$. Since $T^u$ is also a maximal subring of $R$ that does not contain $J$, by Theorem \ref{Max subring classification} $T^u = S' \oplus I'$ for some $S' \in \msS(R)$ and some maximal subideal $I'$ of $J$. It follows that $I'=I$, and $S^u=S'=S^{1+x}$ for some $x \in J$.

As for the bound on the number of conjugates of $T$ in $R$, note that if $y \in I$, then $T^{1+y} = T$. Hence, $1+I$ is contained in the setwise stabilizer of $T$ in $1+J$, and the number of conjugates of $T$ is bounded above by $|1+J : I+I| = |J:I|$.

(3) By Lemma \ref{lem:conj1+J}, $s^{1 + x} = s + (sx - xs)$, and, since $R = S \oplus J$, $S^{1 + x} = S$ implies that $sx = xs$ for all $x \in S$. Since $xy=0=yx$ for all $y \in J$, we get $x \in Z(R)$. Next, assume $s \in Z(S)$.  Then, $sx \in Z(R)$, so for all $s' \in S$,
\begin{equation*}
(s')^{1 + sx} = s' + (s'(sx) - (sx)s') = s',
\end{equation*}
as desired.
\end{proof}







\subsection{Not all conjugates contained in a single subring}

From here, we specialize to the case where $S$ is commutative. In this subsection, we will consider the situation where there does not exist a maximal subring $T \subseteq R$ containing every conjugate of $S$.

\begin{lem}\label{lem:centSin1+J}
Let $S \in \msS(R)$, where $S$ is commutative, and $T = S \oplus I$ be a maximal subring containing $S$. Assume there exists $S' \in \msS(R)$ such that $S' \not\subseteq T$.  If there exists $x \in J$ such that $S^{1+x} = S$, then $x \in I$.  In particular, the setwise stabilizer of $S$ in $1+J$ is contained in $1+I$.
\end{lem}

\begin{proof}
Assume $S^{1+x} = S$ for some $x \in J \backslash I$.  By Theorem \ref{Max subring classification} (2), $I$ is maximal among subideals of $R$ contained in $J$, so as a two-sided ideal of $R$, $J$ is generated by $I$ and $x$.  By Lemma \ref{S and J lem} (3), $x$ is central in $R$, so $J = I + Rx$. In addition, since $R = S \oplus J$ and $J^2=\{0\}$, $Rx=Sx$. Thus, $J = I + Sx$.

Next, since $S' \not\subseteq T$, $S' \neq S$, so there is $x' \in J$ such that $S^{1+x'} = S'$, and we may assume that $x' = y + sx$ for some $y \in I$ and $s \in S$.  By Lemma \ref{lem:1+Jarithmetic} (2), this implies that $1 + x' = (1 + sx)(1 + y)$. Since $S$ is commutative, Lemma \ref{S and J lem} (3) yields
\begin{equation*}
S' = S^{1+x'} = S^{(1 + sx)(1+y)} = S^{1+y} \subseteq T,
\end{equation*}
a contradiction.  Therefore, $x \in I$, and the setwise stabilizer of $S$ in $1+J$ is contained in $1 + I$.
\end{proof}

\begin{lem}\label{lem:numberTconj}
Let $S \in \msS(R)$, where $S$ is commutative, and suppose $T = S \oplus I$ is a maximal subring of $R$ containing $S$.  If there exists $S' \in \msS(R)$ such that $S' \not\subseteq T$ and $x, x' \in J$, then $T^{1+x} = T^{1 + x'}$ if and only if $x' \in x + I$.  In particular, $T$ has exactly $|J:I|$ conjugates in $R$. 
\end{lem}

\begin{proof}
By Lemma \ref{S and J lem} (2), the number of conjugates of $T$ in $R$ is at most $|J:I|$. To prove that there are exactly $|J:I|$ conjugates, suppose that $T^{1+x'} = T^{1+x}$ for $x, x' \in J$ for some $x, x' \in J$. We will show that $x-x' \in I$.

We have $T^{1 + (x' - x)} = T^{(1+x')(1 - x)} = T,$ and hence $S^{1 + (x' - x)} \subseteq T$.  By Lemma \ref{S and J lem} (1), there exists $y \in I$ such that $S^{1 + (x'-x)} = S^{1+y}$, which in turn means that $S^{1 + (x' - x) - y} = S^{(1 + (x' -x))(1 - y)} = S$.  By Lemma \ref{lem:centSin1+J}, $(x' - x) - y \in I$, and so $x' - x \in I$.  
\end{proof}

\begin{lem}\label{lem:intTs}
Let $S \in \msS(R)$, where $S$ is commutative, and let $T = S \oplus I$ be a maximal subring of $R$ containing $S$.  Suppose that not all conjugates of $S$ in $R$ are contained in $T$. Let $T_1 \colonequals T$, $T_2$, \dots, $T_n$ be the conjugates of $T$ in $R$, where $n = |J:I|$, and let $A \colonequals T_1 \cap T_2$.  Then, the following are equivalent for a given $s \in S$:
 \begin{itemize}
  \item[(i)] $s + y \in A$, for some $y \in I$;
  \item[(ii)] $s \in A$;
  \item[(iii)] $s + I \subseteq A$;
  \item[(iv)] $s^{1 + J} \subseteq s + I$;
  \item[(v)] $s + I \subseteq \bigcap_{i = 1}^n T_i$.
 \end{itemize}
In particular, 
\begin{equation*}
\bigcap_{i = 1}^n T_i = A.
\end{equation*}
\end{lem}

\begin{proof}
By Lemma \ref{S and J lem} (2), $T_2 = T^{1+x} \oplus I$ for some $x \in J$. We are assuming that $T_2 \ne T$, so $x \notin I$. Hence, $A = (S^{1+x}  \oplus I) \cap (S \oplus I)$.

Clearly, (v) implies (i), and conditions (i), (ii), and (iii) are equivalent because $I \subseteq A$. We will first show that (ii) implies (iv).  Assume $s \in A$.  By Lemma \ref{lem:conj1+J}, it suffices to show that $s\alpha - \alpha s \in I$ for all $\alpha \in J$. Let $\langle x \rangle$ denote the two-sided ideal of $R$ generated by $x$. Because $I$ is a maximal subideal of $J$, we have $J = I + \langle x \rangle$. Note that $s(\alpha+\alpha') - (\alpha+\alpha')s = (s\alpha - \alpha s) + (s\alpha' - \alpha's)$ for any $\alpha, \alpha' \in J$. So, it will be enough to show that $s\alpha - \alpha s \in I$ for all $\alpha \in \langle x \rangle$. Furthermore, since $J^2=\{0\}$, a generic element of $\langle x \rangle$ has the form $\sum_{j} b_j x c_j$ for some $b_j, c_j \in S$. Thus, (iv) will hold as long as $s(bxc) - (bxc)s \in I$ for all $b, c \in S$.

Now, we know that $s \in S^{1+x}  \oplus I$, so $s = t^{1+x} + y$, where $t \in S$ and $y \in I$.  By Lemma \ref{lem:conj1+J}, 
\begin{equation*}
s =t^{1 + x} + y = t + (tx - xt) + y.
\end{equation*}
Since $R = S \oplus J$, $t = s$ and $tx - xt + y =0$. Hence, $sx-xs=tx-xt = -y \in I$. Let $b, c \in S$. Using the fact that $S$ is commutative and $sx-xs=-y$, we have
\begin{equation*}
s(bxc) - (bxc)s = b(sx)c  - b(xs)c = b(sx - xs)c = -byc \in I,
\end{equation*}
as desired. Therefore, (ii) implies (iv).

To complete the proof that (i)--(v) are equivalent, we will show that (iv) implies (v). Let $s \in S$ and assume that $s^{1+J} \subseteq s + I$. By Lemma \ref{S and J lem} (2), for each $1 \le i \le n$, there exists $x_i \in J$ such that $T_i = T^{1+x_i} = S^{1+x_i} \oplus I$. Let $y \in I$. Then, $s + y = s^{1+x_i} +y - (sx_i - x_is)$. By assumption, $s^{1+x_i} \in s+I$, so $sx_i - x_is \in I$. Thus, for each $i$, $s+y \in T_i$ and $s+I \subseteq T_i$.

To conclude that $\bigcap_{i = 1}^n T_i = A$, it suffices to show that $A \subseteq \bigcap_{i = 1}^n T_i$. Let $a \in A$; then, $a=s+y$ for some $s \in S$ and $y \in I$ because $A \subseteq T$. By the equivalence of (i) and (v), $a \in \bigcap_{i = 1}^n T_i$.
\end{proof}

\begin{prop}\label{prop:upperboundmultT}
Let $S \in \msS(R)$, where $S$ is commutative, and let $T = S \oplus I$ be a maximal subring of $R$ containing $S$.  Suppose that not all conjugates of $S$ in $R$ are contained in $T$. Let $T_1 \colonequals T$, $T_2$, \dots, $T_n$ be the conjugates of $T$ in $R$, where $n = |J:I|$, and let $A \colonequals T_1 \cap T_2$. Then, there exists a maximal subring $M$ of $R$ containing $A + J$, and $\{T_1, T_2, \dots , T_n, M\}$ is a cover of $R$.  Consequently, $\sigma(R) \le |J:I| + 1$.
\end{prop}

\begin{proof}
Note first that $A$ is a proper subring of $T$.  By Lemma \ref{lem:intTs}, $A = \bigcap_{i = 1}^n T_i$, and, for any $x \in J$,
\begin{equation*}
A^{1 + x} = \left(\bigcap_{i=1}^n T_i \right)^{1+x} = \bigcap_{i = 1}^n T_i^{1 + x} = \bigcap_{i = 1}^n T_i = A.
\end{equation*}
So, $A + J$ is a proper subring of $R$, and there is maximal subring $M$ of $R$ containing $A + J$.

We will now show that $\{T_1, T_2, \dots , T_n, M\}$ is a cover of $R$.  Let $r \in R$, and write $r = s + x$, where $s \in S$ and $x \in J$. If $s \in A$, then $r \in A+J \subseteq M$. So, assume that $s \notin A$. By Lemma \ref{lem:intTs}, $s^{1+J} \not\subseteq s+I$, so there exists $\alpha \in J$ such that $s\alpha - \alpha s \notin I$. Since $I$ is maximal in $J$, $J = I + \langle s\alpha - \alpha s \rangle$. So, there exist $y \in I$, $m \ge 1$ and elements $b_j, c_j \in S$ ($1 \le j \le m$) such that $x = y + \sum_{j=1}^m b_j (s \alpha - \alpha s) c_j$. Recalling that $S$ is commutative, we have
\begin{equation*}
s^{1 + b_j \alpha c_j} = s + (s b_ j \alpha c_j - b_j \alpha c_j s)= s + b_j(s\alpha -\alpha s)c_j,
\end{equation*}
which implies that
\begin{equation*}
(s + y)^{\prod_{j=1}^m (1 + b_j \alpha c_j )} = s + \sum_{j=1}^m b_j (s\alpha - \alpha s)c_j + y = s + x,
\end{equation*}
and so $s + x$ is in a conjugate of $T$. Thus, $r \in \bigcup_{i=1}^n T_i$, and $\{T_1, T_2, \dots , T_n, M\}$ is a cover of $R$.
\end{proof}

We now have everything we need in the case where no maximal subring in $R$ contains all complements of $J$.  

\begin{thm}\label{thm:Rcase1}
Suppose $R/J$ is commutative, $\sigma(R) < \sigma(R/J)$, and there does not exist a maximal subring $M$ of $R$ such that 
\[\bigcup_{S \in \msS(R)} S \subseteq M.\]
Then, $\sigma(R) = |J:I| + 1$, where $T = S \oplus I$ is a maximal subring of $R$ with $|J:I|$ minimal. In particular, $\sigma(R)$ has the form $p^d+1$ in this case.
\end{thm}

\begin{proof}
Since $\sigma(R) < \sigma(R/J)$, by Lemma \ref{Sigma elementary lemma}, for each $S \in \msS(R)$, there is a maximal subring $T_S$ with $J \not\subseteq T_S$ such that $T_S = S \oplus I_S$ for some ideal $I_S$ of $R$ contained in $J$.

Because $\bigcup_{S \in \msS(R)} S$ is not contained in any such $T_S$, there are exactly $|J:I_S|$ conjugates of $T_S$ by Lemma \ref{lem:numberTconj}.  If $I = I_S$ is chosen such that $|J:I|$ is minimal, this means there must be at least $|J:I|$ maximal subrings in any cover.  However, since $0 \in \bigcup_{S \in \msS(R)} S$, if $B \colonequals \bigcap_{S \in \msS(R)} S$, $B \neq \varnothing$, and so these $|J:I|$ maximal subrings contain at most
\[ \left(|S|\cdot|I| - |B| \right) \cdot \frac{|J|}{|I|} + |B| = |R| - |B|\left(\frac{|J|}{|I|} - 1 \right) < |R|\]
elements of $R$, and so $\sigma(R) \ge |J:I| + 1$.  On the other hand, by Proposition \ref{prop:upperboundmultT}, $\sigma(R) \le |J:I| + 1$.  The result follows.
\end{proof}

\subsection{All conjugates contained in a single subring}

We will maintain previous notation and assume that $R$ is a finite ring with unity in characteristic $p$ with radical $J$ satisfying $J^2 = \{0\}$ and
\begin{equation*}
\msS(R) = \{ S \subseteq R : S \text{ is a subalgebra of } R, S \cong R/\msJ(R), \text{ and } R = S \oplus \msJ(R)\}.
\end{equation*}
Moreover, we will assume throughout this subsection that there exists a maximal subring $T$ of $R$ such that \[ \bigcup_{S \in \msS(R)} S \subseteq T.\]  We will further define 
\[ A \colonequals \langle S : S \in \msS(R)\rangle, \]
i.e., $A$ is the subring generated by all complements of $J$ in $R$ and is thus the smallest subring containing all $S \in \msS(R)$.  Define $K \colonequals \msJ(A)$, and note that $K$ is itself an ideal of $R$ and that $A = S \oplus K$ for $S \in \msS(R)$.  

\begin{lem}\label{lem:case2commZ}
Suppose $R/J$ is commutative, $\sigma(R) < \sigma(R/J)$, and $\bigcup_{S \in \msS(R)} S \subseteq T$ for some proper subring $T$, where $T = S \oplus I$.  Then, $J = I \oplus \overline{I}$, where $\overline{I}$ is an ideal of $R$ contained in $Z \colonequals J \cap Z(R)$.  In particular, if $A$ is the algebra generated by the complements of $J$ in $\msS(R)$ and $K = \msJ(A)$, then $J = K \oplus \xbar{K}$, where $\xbar{K}$ is an ideal of $R$ contained in $J \cap Z(R)$. 
\end{lem}

\begin{proof}
Suppose each complement in $\msS(R)$ is contained in $T = S \oplus I$.  We claim first that there exists $x \in J \backslash I$ such that $x$ centralizes $S$.  Let $x' \in J \backslash I$.  Since all conjugates of $S$ are contained in $T$, $T^{1 + x'} = S^{1 + x'} \oplus I = S \oplus I$.  Thus, $S^{1 + x'} \subseteq T$, and, by Lemma \ref{S and J lem} (1), there is $y \in I$ such that $S^{1 + x'} = S^{1 + y}$, and so $S^{1 + (x' - y)} = S$.  By Lemma \ref{S and J lem} (3), for all $s \in S$, $s(x'-y) = (x' -y)s$.  Since $x' \not\in I$, $x = x' - y \not\in I$, and $x$ is the desired element that centralizes $S$.

Now, since $x$ centralizes $S$ and $J^2 = \{0\}$, we have $x \in J \cap Z(R) = Z$.  This means that $J = I + Z$, and $Z$ is an ideal of $R$ since $S$ is commutative. Because $S$ is semisimple, $Z$ is semisimple as both a left $S$-module and a right $S$-module.  Since $I \cap Z$ is a submodule of $Z$, there exists a left $S$-module $I_\ell$ in $Z$ such that $Z = (I \cap Z) \oplus I_\ell$ (a direct sum of left $S$-modules), and similarly there exists a right $S$-module $I_r$ in $Z$ such that $Z = (I \cap Z) \oplus I_r$ (a direct sum of right $S$-modules). But, $S \oplus Z$ is a commutative ring, so $\overline{I} \colonequals I_\ell = I_r$, and hence $I \cap Z$ has an ideal complement in $Z$.  Hence, $J = I + Z = I \oplus \overline{I}$, as desired.

Finally, since $\bigcup_{S \in \msS(R)} S \subseteq A$, the result can be applied to $T = A$.
 \end{proof}
 
\begin{lem}
\label{lem:SplusKbar}
Suppose $R/J$ is commutative and $A \colonequals \langle S : S \in \msS(R)\rangle$ is a proper subring of $R$ with $K = \msJ(A) \neq \{0\}$ and $J = K \oplus \xbar{K}$.  If $R$ is $\sigma$-elementary and $\mathcal{C}$ is a minimal cover of $R$, then, for every $S \in \msS(R)$, $S \oplus \xbar{K}$ is contained in some $T \in \mathcal{C}$. 
\end{lem}
 
\begin{proof}
 Let $S \in \msS(R)$ and assume that $S \oplus \xbar{K}$ is not contained in any $T \in \mathcal{C}$.  We note that $S \oplus \xbar{K}$ is itself a subring of $R$, since, by Lemma \ref{lem:case2commZ}, $\xbar{K}$ is an ideal of $R$.  Then, 
 \[ \xbar{\mathcal{C}} \colonequals \{T \cap (S \oplus \xbar{K}) : T \in \mathcal{C} \}\]
 is a cover of $S \oplus \xbar{K}$, so $\sigma(R) \ge \sigma(S \oplus \xbar{K})$.  However, $S \oplus \xbar{K} \cong R/K$ and $K \neq \{0\}$, so $\sigma(R) \ge \sigma(S \oplus \xbar{K}) = \sigma(R/K)$, a contradiction to $R$ being $\sigma$-elementary.  Therefore, for every $S \in \msS(R)$, $S \oplus \xbar{K}$ is contained in some $T \in \mathcal{C}$, as desired.
\end{proof}

\begin{thm}
\label{thm:Rcase2}
 Suppose $R/J$ is commutative and $A \colonequals \langle S : S \in \msS(R)\rangle$ is a proper subring of $R$ with $K = \msJ(A)$.  If $R$ is $\sigma$-elementary, then $K = \{0\}$, i.e., $R$ is a commutative ring.
\end{thm}

\begin{proof}
Assume $R$ is $\sigma$-elementary, and suppose on the contrary that $K \neq \{0\}$. Let $\mathcal{C}$ be a minimal cover of $R$.  By Lemma \ref{lem:case2commZ}, $J = K \oplus \xbar{K}$, where $\xbar{K}$ is an ideal of $R$ contained in $J \cap Z(R)$, and by Lemma \ref{lem:SplusKbar}, for every $S \in \msS(R)$, $S \oplus \xbar{K}$ is contained in some $T \in \mathcal{C}$.  
 
Let $S \in \msS(R)$.  Note that no proper subring of $R$ that contains $\xbar{K}$ can contain all conjugates of $S$ in $R$, since $R = A \oplus \xbar{K}$.  Moreover, all subrings of the form $S^\prime \oplus \xbar{K}$ are conjugate in $R$ (since $\xbar{K}$ is an ideal and the subrings in $\msS(R)$ are conjugate in $R$), so we can refer to these subrings as the conjugates of $S \oplus \xbar{K}$ in $R$.  If $T$ is a maximal subring of $R$ containing both $\xbar{K}$ and as many conjugates of $S$ as possible, then $T = B \oplus \xbar{K}$, where $B$ is a maximal subring of $A$ containing as many conjugates of $S$ as possible (and $|A:B|$ minimal subject to that condition).  

Without loss of generality, assume $S \subseteq B$, and so $B = S \oplus \msJ(B)$.  By construction, $A$ is the smallest subring of $R$ containing all the conjugates of $S$, so no maximal subring of $A$ contains all conjugates of $S$ in $A$. Thus, on the one hand, it requires at least $|A:B| = |K:\msJ(B)|$ maximal subrings of $R$ to cover all of the conjugates of $S \oplus \xbar{K}$ in $R$.  On the other hand, by Lemma \ref{lem:numberTconj}, $|K:\msJ(B)|$ maximal subrings of $A$ suffice to cover all of the conjugates of $S$ in $A$, and hence $|K:\msJ(B)|$ maximal subrings of $R$ suffice to cover all of the conjugates of $S \oplus \xbar{K}$ in $R$.


Let $\mathcal{B}$ be a minimal cover of the conjugates of $S \oplus \xbar{K}$ in $R$, i.e., assume that $\mathcal{B}$ contains $|K:\msJ(B)|$ maximal subrings.  Then,
\begin{equation*}
\xbar{\mathcal{B}} \colonequals \{ T \cap A : T \in \mathcal{B}\}
\end{equation*}
is a minimal cover of the conjugates of $S$ in $A$.  However, by the same reasons as in the proof of Theorem \ref{thm:Rcase1}, $\xbar{\mathcal{B}}$ cannot be a cover of $A$: $D \colonequals \bigcap_{C \in \xbar{\mathcal{B}}} C$ contains at least two elements ($0$ and $1$), so the collection $\xbar{\mathcal{B}}$ contains at most
\[ \left(|B| - |D| \right) \cdot \frac{|K|}{|\msJ(B)|} + |D| = |A| - |D|\left(\frac{|K|}{|\msJ(B)|} - 1 \right) < |A|\]
elements.  This implies that $\mathcal{B}$ cannot be a cover of $R = A \oplus \xbar{K}$, and hence $\sigma(R) \ge |K:\msJ(B)| + 1$.  On the other hand, by Proposition \ref{prop:upperboundmultT}, $|K:\msJ(B)| + 1$ subrings suffice to cover $A \cong R/\xbar{K}$, so $\sigma(R) \ge \sigma(R/\xbar{K})$, a contradiction to $R$ being $\sigma$-elementary.  Therefore, we conclude that $K = \{0\}$, and hence $J = \xbar{K} \subseteq Z(R)$, and $R$ is commutative.
\end{proof}

\subsection{Proofs of Theorem \ref{thm:main} and Corollary \ref{cor:13}}

We are now able to prove Theorem \ref{thm:main}.

\begin{proof}[Proof of Theorem \ref{thm:main}]
 Let $R$ be a finite, $\sigma$-elementary ring with Jacobson radical $J \neq \{0\}$ such that $R/J$ is commutative.  Since $R$ is $\sigma$-elementary, $\sigma(R) < \sigma(R/J)$.  By Theorem \ref{full reduction thm}, we may assume that $R$ has characteristic $p$ for some prime $p$ and $J^2 = \{0\}$.  By Theorem \ref{Wedderburn thm}, there exists a subring $S$ of $R$ such that $R \cong S \oplus J$, and $S$ is unique up to conjugation by elements of $1 + J$. On the one hand, if no maximal subring contains all conjugates of $S$ in $R$, then by Theorem \ref{thm:Rcase1} $\sigma(R) = p^d +1$ for some positive integer $d$.  On the other hand, if there exists a maximal subring containing all conjugates of $S$, then by Theorem \ref{thm:Rcase2}, $R$ is a commutative ring, and $\sigma(R) = p^d + 1$ for some positive integer $d$ by Corollary \ref{Full commutative case}.  The result follows.
\end{proof}

To prove Corollary \ref{cor:13}, we use the formula from \cite{Werner} for the covering number of a direct product of copies of $\F_q$. Let $q = p^n$, define
\begin{equation*}
\omega(n) = \begin{cases} 1 & n=1 \\ \# \text{ prime divisors of } n & n >1 \end{cases}
\end{equation*}
and let $\tau(q)$ be as defined in Theorem \ref{thm:Werner1}. Then, by \cite[Thm.\ 5.3]{Werner},
\begin{equation*}
\sigma\big(\prod_{i=1}^{\tau(q)} \F_q\big) = \tau(q) \omega(n) + n \binom{\tau(q)}{2}.
\end{equation*}

\begin{proof}[Proof of Corollary \ref{cor:13}]
As usual, by Theorem \ref{full reduction thm} we may assume that $R$ is finite of characteristic $p$ and that $J^2=\{0\}$. If $R$ is coverable with $\sigma(R)$ finite, then there exists a $\sigma$-elementary residue ring $\olR$ of $R$ such that $\sigma(R)=\sigma(\olR)$. So, we can assume without loss of generality that $R$ is $\sigma$-elementary. If $J \ne \{0\}$, then by Theorem \ref{thm:main}, $\sigma(R) = p^d+1$ for some $d \ge 1$. Clearly, $\sigma(R) \ne 13$ in this case.

Assume now that $J = \{0\}$. So, $R = R/J$ is commutative. By Theorem \ref{thm:comm_sigma}, $R \cong \prod_{i=1}^{\tau(q)} \F_q$ for some $q=p^n$, and hence 
\begin{equation*}
\sigma(R) = \tau(q) \omega(n) + n \binom{\tau(q)}{2}.
\end{equation*}
If $n=1$, then $\sigma(R) = p + \binom{p}{2}$, which cannot equal 13. If $n \ge 2$, then $\sigma(R) \ge \tau(q)^2$, so to get $\sigma(R)=13$ we need $\tau(q) \le 3$. This means that $\F_p[x]$ contains at most 2 irreducible polynomials of degree $n$. The only cases in which this occurs are $q=4$ or $q=8$. However, one may check that if $q=4$, then $\sigma(R) = 4$; and if $q=8$, then $\sigma(R)=12$. We conclude that it is impossible for $\sigma(R)$ to equal 13.
\end{proof}

\end{document}